\documentclass[a4paper,12pt, british, reqno]{amsart}

\usepackage{babel}
\usepackage{mathabx}
\usepackage{amssymb}
\usepackage{amsfonts}
\usepackage{enumitem}
\usepackage{hyperref}
\usepackage[utf8]{inputenc}
\usepackage{newunicodechar}
\usepackage{mathtools}
\usepackage{varioref}
\usepackage[centering]{geometry}
\usepackage[arrow,curve,matrix]{xy}
\usepackage{csquotes}
\usepackage{dynkin-diagrams}

\usepackage{colortbl}
\usepackage{graphicx}
\usepackage{tikz}
\usepackage{tikz-cd}

\usepackage{mathrsfs}

\usepackage{multirow}
\usepackage{multicol}
\usepackage{longtable} 
\usepackage{caption}
\usepackage{hhline}
\newcolumntype{M}[1]{>{\centering\arraybackslash}m{#1}} %define dimension for long stable

\definecolor{linkred}{rgb}{0.7,0.2,0.2}
\definecolor{linkblue}{rgb}{0,0.2,0.6}

\numberwithin{figure}{section}

\setdescription{labelindent=\parindent, leftmargin=2\parindent}
\setitemize[1]{labelindent=\parindent, leftmargin=2\parindent}
\setenumerate[1]{labelindent=0cm, leftmargin=*, widest=iiii}

\DeclareFontFamily{OMS}{rsfs}{\skewchar\font'60}
\DeclareFontShape{OMS}{rsfs}{m}{n}{<-5>rsfs5 <5-7>rsfs7 <7->rsfs10 }{}
\DeclareSymbolFont{rsfs}{OMS}{rsfs}{m}{n}
\DeclareSymbolFontAlphabet{\scr}{rsfs}
\DeclareSymbolFontAlphabet{\scr}{rsfs}

\DeclareMathOperator{\bir}{bir}

\DeclareMathOperator{\Hom}{Hom}

\DeclareMathOperator{\rank}{rank}

\DeclareMathOperator{\red}{red}

\DeclareMathOperator{\Locus}{Locus}

\newcommand{\sB}{\scr{B}}

\newcommand{\sD}{\scr{D}}

\newcommand{\sK}{\scr{K}}

\newcommand{\sO}{\scr{O}}

\newcommand{\sT}{\scr{T}}
\newcommand{\sU}{\scr{U}}
\newcommand{\sV}{\scr{V}}
\newcommand{\sW}{\scr{W}}

\newcommand{\cD}{\mathcal D}

\newcommand{\cF}{\mathcal F}

\newcommand{\cO}{\mathcal O}

\newcommand{\bG}{\mathbb{G}}

\newcommand{\bP}{\mathbb{P}}

\newcommand{\bR}{\mathbb{R}}

\newcommand{\fg}{\mathfrak{g}}

\theoremstyle{plain}
\newtheorem{thm}{Theorem}[section]

\newtheorem{cor}[thm]{Corollary}

\newtheorem{lem}[thm]{Lemma}

\newtheorem{prop}[thm]{Proposition}

\theoremstyle{remark}

\newtheorem{c-n-d}[thm]{Claim and Definition}

\newtheorem{rem}[thm]{Remark}

\newtheorem*{rem-nonumber}{Remark}

\numberwithin{equation}{thm}

\setlist[enumerate]{label=(\thethm.\arabic*), before={\setcounter{enumi}{\value{equation}}}, after={\setcounter{equation}{\value{enumi}}}}

\newcommand{\factor}[2]{\left. \raise 2pt\hbox{$#1$} \right/\hskip -2pt\raise -2pt\hbox{$#2$}}

\author{Jie Liu} %
\address{Jie Liu, Institute of Mathematics, Academy of Mathematics and Systems Science, Chinese Academy of Sciences, Beijing, 100190, China}
\email{\href{jliu@amss.ac.cn}{jliu@amss.ac.cn}}
\urladdr{\href{http://www.jliumath.com}{http://www.jliumath.com}}


\makeatletter
\@namedef{subjclassname@2020}{2020 Mathematics Subject Classification}
\makeatother

\title[]{Compact K\"ahler contact manifolds}
\date{\today}

\makeatletter

\hypersetup{
	pdfauthor={\authors},
	pdftitle={\@title},
	pdfsubject={\@subjclass},
	pdfkeywords={\@keywords},
	pdfstartview={Fit},
	pdfpagemode={UseOutlines},
	pdfpagelayout={OneColumn},
	colorlinks,
	linkcolor=linkblue,
	citecolor=linkred,
	urlcolor=linkred
}
\makeatother

\begin{document}
	
	\begin{abstract}
        We prove that a non-projective compact K\"ahler contact manifold is of the form $\bP T_Y$, where $Y$ is a compact K\"ahler manifold.
	\end{abstract}

	\maketitle
	\tableofcontents

\section{Introduction}

A complex manifold $X$ of dimension $2n+1$ is said to be a \emph{contact manifold} if there exists an exact sequence of vector bundles:
\[
0\longrightarrow \cD \longrightarrow T_X \longrightarrow L\longrightarrow 0,
\]
where $T_X$ is the tangent bundle of $X$ and $L$ is a line bundle, called the \emph{contact line bundle}, such that the induced map
\[
\wedge^2\cD \longrightarrow T_X/\cD=L,\quad v\wedge v' \longmapsto [v,v']/\cD
\]
is everywhere non-degenerate. There are two basic ways to construct contact manifolds:
\begin{itemize}
	\item For any complex manifold $Y$, the projectivized tangent bundle $\bP T_Y$ carries a natural contact structure \cite[\S\,2.6]{KebekusPeternellSommeseWisniewski2000}.
	
	\item Let $\fg$ be a simple Lie algebra. Then the unique closed orbit in $\bP\fg$ for the coadjoint action is a Fano contact manifold with $b_2=1$ \cite[\S\,2]{Beauville1998}.
\end{itemize}

It is generally conjectured that \emph{every projective contact manifold arises in this manner}. For projective manifolds with $b_2\geq 2$, this statement was fully established in \cite{KebekusPeternellSommeseWisniewski2000} and \cite{Demailly2002} (see also \cite{Druel1998} for the low-dimensional case). As no compact contact manifold outside the aforementioned class is currently known to exist, it is natural to ask whether the projectivity assumption is essential. In this short note, we investigate the K\"ahler case and prove the following result:

\begin{thm}
	\label{thm.MainThm}
	Let $(X,L)$ be a compact K\"ahler contact manifold, which is not projective. Then there exists a compact K\"ahler manifold $Y$ such that $(X,L)$ is isomorphic to $(\bP T_Y, \cO_{\bP T_Y}(1))$.
\end{thm}

This result was already established for K\"ahler threefolds in \cite[Theorem 4.1]{Peternell2001a}. To prove Theorem \ref{thm.MainThm}, we follow the same strategy as in \cite{KebekusPeternellSommeseWisniewski2000} for the projective setting. Roughly speaking, the arguments in \cite{KebekusPeternellSommeseWisniewski2000,Demailly2002} rest on three core ingredients: Demailly's Frobenius integrality theorem \cite{Demailly2002}, the Ionescu–Wi\'sniewski inequality \cite[IV, \S\,2]{Kollar1996}, and the contraction theorem in Mori theory. 

In the K\"ahler setting, Demailly’s integrality theorem holds true, and the Ionescu--Wi'sniewski inequality can be derived using the same arguments as for projective manifolds, with suitable adaptations via Douady/Barlet spaces (see \S\,\ref{s.IW-inq} for details). However, the contraction theorem for general K\"ahler spaces remains open, except in dimensions $\leq 3$. On the other hand, combining results from \cite{Demailly2002} and \cite{Ou2025} implies that every compact K\"ahler contact manifold $X$ is uniruled. We therefore consider the maximally rationally connected quotient of $X$, instead of appealing to the contraction theorem. This allows us to apply a result of C.~Araujo \cite{Araujo2006,AraujoDruelKovacs2008} to construct a dominant family of projective spaces contained in $X$, which we then show induces a projective bundle structure of $X$.

Throughout this paper, for a compact complex space $X$, we will denote by $\sB(X)$ and $\sD(X)$ its Barlet space and Douady space, respectively. For a vector bundle $E$, we denote by $\bP E$ its projectivization in the sense of Grothendieck.

    \subsection*{Acknowledgments}   
I would like to thank Wenhao Ou for useful discussions and Andreas H\"oring for his remarks on the first draft. The author is supported by the National Key Research and Development Program of China (No. 2021YFA1002300), the Youth Innovation Promotion Association CAS, the CAS Project for Young Scientists in Basic Research (No. YSBR-033) and the NSFC grant (No. 12288201).
	
\section{Ionescu--Wi\'sniewski inequality}
\label{s.IW-inq}

The proof of the Ionescu--Wiśniewski inequality in the projective case, as given in \cite[IV, §2]{Kollar1996}, can be adapted to the Kähler setting using the Moishezonness results for Barlet spaces from \cite{Fujiki1982}. However, since we have not been able to locate this statement in the literature, we provide a detailed proof for the reader's convenience. 

Recall that a compact complex space $X$ is of \emph{Fujiki class} (resp. \emph{Moishezon}) if the underlying reduced complex space $X_{\red}$ is bimeromorphic to a compact K\"ahler manifold (resp. a projective manifold).

\begin{prop}[\protect{\cite[(4.3), Proposition 4]{Fujiki1982}}]
	\label{prop.Def-Moishezon}
	Let $X$ be a reduced compact complex space of Fujiki class. Let $Z$ be an irreducible and reduced closed analytic subspace of $\sD(X)$ such that the general fiber of the universal family $U_Z\rightarrow Z$ is reduced and irreducible. Then for any point $x\in X$, the reduced analytic subspace $Z_x$ of $Z$ parametrizing cycles in $Z$ passing through $x$ is Moishezon. 
\end{prop}

\begin{proof}
	By \cite[Corollary]{Fujiki1984}, the complex space $Z$ is compact. Then \cite[(4.3), Proposition 4]{Fujiki1982} says that the evaluation morphism $e_Z\colon U_Z\rightarrow X$ is Moishezon, so the analytic subspace $e_Z^{-1}(x)$ is a Moishezon space by \cite[(1.7)]{Fujiki1982}. Hence the image $Z_x=q(e_Z^{-1}(x))$ is Moishezon by \cite[(1.6), Proposition 1]{Fujiki1984}, where the map $q\colon U_Z\rightarrow Z$ is the natural projection.
\end{proof}

Let $X$ be a compact K\"ahler manifold. Then the space $\Hom(\bP^1,X)$ of holomorphic maps $f\colon \bP^1\rightarrow X$ is an open subset of $\sD(\bP^1\times X)$. Let $\Hom_{\bir}(\bP^1,X)$ be the open subset corresponding to the holomorphic maps $f\colon \bP^1\rightarrow X$ which are birational onto their image.  Then there exist holomorphic maps
\begin{equation}
	\label{eq.morp-cycle}
	u\colon \Hom_{\bir}(\bP^1,X) \longrightarrow \sD(\bP^1\times X) \longrightarrow \sB(\bP^1\times X) \xlongrightarrow{p_*} \sB(X),
\end{equation}
where the first map is the natural open embedding, the second map is the Douady--Barlet map, and the last one is the push-forward map associated to the second projection $p\colon \bP^1\times X\rightarrow X$ \cite{Campana1980}. 

Let $\sV$ be an irreducible component of $\Hom_{\bir}(\bP^1,X)$. Denote by $\Locus \sV$ the image of $\sV\times \bP^1$ in $X$ under the evaluation map $\Hom_{\bir}(\bP^1,X)\times \bP^1\rightarrow X$. For any points $x,y\in X$, we will denote by $\sV_x$ (resp. $\sV(0\mapsto x)$, or $\sV(0\mapsto x, \infty\mapsto y)$) the closed analytic subspace of $\sV$ consisting of maps $f\colon \bP^1\rightarrow X$ such that $x\in f(\bP^1)$ (resp. $f(0)=x$, or $f(0)=x$ and $f(\infty)=y$). The loci $\Locus \sV(0\mapsto x)$ and $\Locus \sV(0\mapsto x, \infty\mapsto y)$ are defined in a similar way. We call $\sV$ an \emph{unsplit family of rational curves} if $u(\sV)$ is compact and a \emph{generically unsplit family of rational curves} if $u(\sV_x)$ is compact for general points $x\in \Locus \sV$. 

We need the following “bend-and-break lemma” for compact Kähler manifolds. Thanks to the Moishezon property of Douady spaces given in Proposition \ref{prop.Def-Moishezon},  one can adopt the proof in the projective case \cite[\S\,II.5]{Kollar1996} to the Kähler case.

\begin{prop}[\protect{\cite[II, Theorem 5.4 and Corollary 5.6]{Kollar1996}}]
	\label{prop.Bend-Break}
	Let $X$ be a compact K\"ahler manifold and let $\sV$ be a closed irreducible analytic subspace of $\Hom_{\bir}(\bP^1,X)$. If there exist points $x,y\in X$ such that $\dim \sV(0\mapsto x, \infty \mapsto y)\geq 2$, then $u(\sV_x)$ is not compact.  
\end{prop}

\begin{proof}
	We follow the arguments in \cite{Kollar1996}, pointing out only the necessary modifications for the K\"ahler case.  The group of automorphisms of $\bP^1$ fixing two points is the multiplicative group $\bG_m$. Since $\sV(0\mapsto x, \infty\mapsto y)$ is Moishezon by Proposition \ref{prop.Def-Moishezon} (see \cite[(1.7)]{Fujiki1982}), we can find a one-dimensional closed irreducible analytic subspace of $\sV(0\mapsto x,\infty\mapsto y)$ that is not contained in its $\bG_m$-orbit. Let $T$ be its normalization and let $\widebar{T}$ be a smooth compactification of $T$. Now we construct a commutative diagram as follows:
	\[
	\begin{tikzcd}[row sep=large, column sep=large]
		\bP^1\times T  \arrow[r,hookrightarrow] 
		    & Y \arrow[r] \arrow[d] \arrow[dl]
		        &  S \arrow[d,hook]\\
		\widebar{T}
		    &   X\times \widebar{T}  \arrow[r] \arrow[l]
		        & X,
	\end{tikzcd}
	\]
	where $Y$ is the normalization of the closure of the image of the natural map $\bP^1\times T\rightarrow X\times \widebar{T}$ and $S$ is the closure of the image of the evaluation map $\bP^1\times T\rightarrow X$. 
	
	Assume to the contrary that $u(\sV_x)$ is compact so that $u(\sV_x)$ is a closed subset of $\sB(X)$. Then the fibers of $Y\rightarrow \widebar{T}$ are irreducible and reduced. Note that $Y\rightarrow \widebar{T}$ is flat \cite[III, Proposition 9.7]{Hartshorne1977} and therefore the arithmetic genus of the fibers of $Y\rightarrow \widebar{T}$ is constant and hence equal to $0$. The fibers of $Y\rightarrow \widebar{T}$ are therefore smooth rational curves, so $Y\rightarrow \widebar{T}$ is a ruled surface. One observes that $Y\rightarrow \widebar{T}$ admits two disjoint sections contracted to points in $S$. In particular, by the Negativity Lemma, the argument in \cite[II, Theorem 5.4]{Kollar1996} then yields a contradiction.
\end{proof}

 Given a compact K\"ahler manifold $X$, a line bundle $L$ on $X$ and an irreducible closed subset $\sV$ of $\Hom_{\bir}(\bP^1,X)$, we denote $\deg f^*L$ by $L\cdot \sV$, where $f\in \sV$. Clearly, the number $L\cdot \sV$ is independent of the choice of $f$. Following the argument of \cite[IV, Corollary 2.6]{Kollar1996} verbatim, we apply Proposition~\ref{prop.Bend-Break} to obtain the following Ionescu--Wi\'sniewski inequality for compact K\"ahler manifolds. 

\begin{prop}
	\label{prop.IW-Ineq}
	Let $X$ be a compact K\"ahler manifold and let $\sV\subset \Hom_{\bir}(\bP^1,X)$ be a generically unsplit family of rational curves. Then for any point $x\in \Locus\sV$ such that $u(\sV_x)$ is compact, we have
	\[
	\dim X + (-K_X\cdot \sV) \leq \dim \Locus \sV + \dim \Locus\sV(0\mapsto x) + 1.
	\]
\end{prop}

\begin{rem}
	The statements of Propositions \ref{prop.Bend-Break} and \ref{prop.IW-Ineq} also hold for compact complex manifolds of Fujiki class.
\end{rem}

\section{Deformation of extremal rational curves}
\label{s.Def-Rat-Curve}

Instead of relying on the contraction theorem from Mori's theory as used in \cite[§\,2.5]{KebekusPeternellSommeseWisniewski2000}, which has not been established in the K\"ahler setting, we will apply the deformation theory of rational curves, combined with the recently established cone theorem for Kähler spaces, to construct a generic projective bundle structure on compact K\"ahler contact manifolds.

\subsection{Extremal rational curves}
\label{ss.Extremal-Rat-Cur}

Mori's cone theorem has recently been generalized to K\"ahler spaces. More precisely, let $X$ be a compact K\"ahler manifold. Let $\widebar{NA}(X)\subset N_1(X)$ be the closed cone generated by the classes of positive closed currents (see \cite[\S\,3.2]{HoeringPeternell2016}). According to \cite[Theorem 0.5]{HaconPaun2024} and \cite[Theorem 1.1]{Ou2025} (see also \cite[Theorem 1.2]{HoeringPeternell2016} for the $3$-dimensional case),  there exist at most countably many rational curves $\{\Gamma_i\}_{i\in I}$ such that
\[
\widebar{NA}(X) = \widebar{NA}(X)_{K_X\geq 0} +\sum_{i\in I} \bR^+[\Gamma_i].
\]
Let $C\subset X$ be an irreducible rational curve. We call $C$ an \emph{extremal rational curve} if there exists $i\in I$ such that $[C]\in \bR^+[\Gamma_i]$ and $-K_X\cdot C=\min \{-K_X\cdot \Gamma\}$, where $\Gamma$ runs among all irreducible rational curves with $[\Gamma]\in \bR^+[\Gamma_i]$. 

\begin{lem}
	\label{lem.unsplit}
	Let $C$ be an extremal rational curve with $f\colon \bP^1\rightarrow X$ its normalization. Let $\sV\subset \Hom_{\bir}(\bP^1,X)$ be an irreducible component containing $[f]$. Then $\sV$ is an unsplit family of rational curves.
\end{lem}

\begin{proof}
	Assume to the contrary that $\sV$ is not unsplit; that is, $u(\sV)$ is not compact. Since the irreducible components of $\sB(X)$ are compact by \cite[Corollary]{Fujiki1984}, the image $u(\sV)$ is not closed in $\sB(X)$. Let $[C']$ be a $1$-cycle in $\widebar{u(\sV)}\setminus u(\sV)$. Then the cycle $[C']$ is non-integral with rational components $C'_i$ by \cite[II, Proposition 2.2]{Kollar1996}. Nevertheless, since $\bR^+[C]$ is an extremal ray of $\widebar{NA}(X)$ and $[C']\in \bR^+[C]$, one gets $[C'_i]\in \bR^+[C]$. This implies
	\[
	-K_X\cdot C'_i<-K_X\cdot C'=-K_X\cdot C,
	\]
	which contradicts the minimality of $-K_X\cdot C$, so the result follows.
\end{proof}

As an application of the Ionescu--Wi\'sniewski inequality, we obtain the following statement on the length of extremal rays on compact K\"ahler manifolds, generalizing a well-known result in the projective case \cite{ChoMiyaokaShepherd-Barron2002}. 

\begin{prop}
	\label{prop.max-length}
	Let $X$ be a compact K\"ahler manifold, and let $C$ be an extremal rational curve on $X$. Then $-K_X\cdot C\leq \dim X+1$, and the equality holds if and only if $X$ is isomorphic to a projective space.
\end{prop}

\begin{proof}
	Let $C$ be an extremal rational curve and let $\sV\subset \Hom_{\bir}(\bP^1,X)$ be an irreducible component containing the normalization $f\colon \bP^1\rightarrow C$. Then $\sV$ is unsplit. In particular, the Ionescu--Wi\'sniewski inequality from Proposition \ref{prop.IW-Ineq} implies
	\[
	\dim X + (-K_X\cdot C) \leq 2\dim X + 1,
	\]
	so $-K_X\cdot C\leq \dim X+1$, and the equality holds only if $\dim \Locus \sV(0\mapsto x)=\dim X$ for any point $x\in \Locus \sV$; that is, $\Locus \sV(0\mapsto x)=X$. In particular, the manifold $X$ is rationally connected and hence projective. Finally it follows from \cite[Main Theorem 0.1]{ChoMiyaokaShepherd-Barron2002} that $X$ is isomorphic to a projective space in this case.
\end{proof}

\subsection{Rational curves on contact manifolds}

Let $(X,L)$ be a non-projective compact K\"ahler contact manifold of dimension $2n+1$. Then $-K_X=(n+1)L$ (see \cite[\S\,2]{KebekusPeternellSommeseWisniewski2000}). By \cite[Corollary 2]{Demailly2002}, the canonical bundle $K_X$ and hence the dual contact line bundle $L^{-1}$ is not pseudoeffective. In particular, the manifold $X$ is uniruled by \cite[Theorem 1.1]{Ou2025} and by the cone theorem, there exists an extremal rational curve $C$ on $X$. Let $\sV\subset \Hom_{\bir}(\bP^1,X)$ be an irreducible component containing the normalization $f\colon \bP^1\rightarrow C$. Then $\sV$ is an unsplit family of rational curves by Lemma \ref{lem.unsplit}.

Let $E\rightarrow \bP^1$ be a vector bundle over a rational curve. Recall that Grothendieck's theorem implies that $E$ decomposes into a sum of line bundles, $E\cong \oplus \cO_{\bP^1}(e_i)$. Denote by $\rank^+(E)$ the number of positive entries in the splitting type; that is, $\rank^+(E)=\sharp\{e_i>0\}$.

\begin{lem}[\protect{\cite[Propositions 2.8 and 2.9]{KebekusPeternellSommeseWisniewski2000}}]
	\label{lem.Locus-Rational-Curves}
	Let $(X,L)$ be a compact K\"ahler contact manifold of dimension $2n+1$.
	\begin{enumerate}
		\item Let $f\colon \bP^1\rightarrow C\subset X$ be the normalization of a rational curve such that $\deg f^*L=1$. Then $\rank^+ f^*T_X=n$.
		
		\item Let $\sV$ be an unsplit family of rational curves on $X$ such that $L\cdot \sV=1$. Then $\Locus \sV=X$ and $\Locus \sV(0\mapsto x)=n$ for any $x\in X$.
	\end{enumerate}
\end{lem} 

\begin{proof}
	The first statement proved in \cite[Proposition 2.8]{KebekusPeternellSommeseWisniewski2000}. The second statement was proved for projective manifolds in \cite[Proposition 2.9]{KebekusPeternellSommeseWisniewski2000}, but the same argument still holds in the K\"ahler case by applying Proposition \ref{prop.IW-Ineq}.
\end{proof}

\begin{lem}
	\label{lem.Degree-Locus}
	$\deg f^*L=1$, $\Locus \sV=X$ and $\dim \Locus\sV(0\mapsto x)=n$ for any $x\in X$.
\end{lem}

\begin{proof}
	Let $x\in \Locus \sV$ be an arbitrary point. By the Ionescu--Wi\'sniewski inequality (see Proposition \ref{prop.IW-Ineq}), we have
	\begin{align*}
		\dim X + (n+1) \deg f^*L &  \leq \dim \Locus \sV + \dim \Locus\sV(0\mapsto x) +1 \\
		& \leq \dim X + \dim \Locus \sV(0\mapsto x) +1.
	\end{align*}
	In particular, if $\deg f^*L\geq 2$, then one obtains
	\[
	\dim X = 2n+1 \leq (n+1) \deg f^*L -1 \leq  \dim \Locus \sV(0\mapsto x).
	\]
	It follows that $X=\Locus\sV(0\mapsto x)$, which implies that $X$ is rationally connected and hence projective. This contradicts our assumption. Hence $\deg f^*L=1$ and the result then follows from Lemma \ref{lem.Locus-Rational-Curves}.
\end{proof}

\subsection{MRC quotient}

Let $(X,L)$ be a non-projective compact K\"ahler contact manifold of dimension $2n+1$ so that $X$ is uniruled. Let $h\colon X\dashrightarrow T$ be the \emph{maximally rationally connected quotient} (MRC quotient for short). Then $1\leq m\coloneqq \dim T\leq 2n$ since $X$ is non-projective. Let $\sV\subset \Hom_{\bir}(\bP^1,X)$ be an irreducible component containing an extremal rational curve on $X$. Then $\sV$ is an unsplit family of rational curves such that $L\cdot \sV=1$, $\Locus \sV=X$ and $\dim \Locus \sV(0\mapsto x)=n$ for any point $x\in X$ by Lemma \ref{lem.Degree-Locus}. We need the following result, which follows from the proof of \cite[Proposition 2.11]{KebekusPeternellSommeseWisniewski2000}.

\begin{lem}
	\label{lem.Contact-Fiber}
	Let $(X,L)$ be a contact complex manifold, and let $\phi\colon X\rightarrow Y$ be a surjective holomorphic map to a complex manifold $Y$ with dimension $m$.  Let $F$ be a complex submanifold of $X$ such that $\phi(F)$ is a point. Then there exists a map of vector bundles as follows:
	\begin{equation}
		\label{eq.beta}
		\beta\colon L^{\oplus m} \longrightarrow T_X|_F
	\end{equation}
	If we assume furthermore that $\phi$ has maximal rank at some point of $F$, then $\beta$ vanishes identically only if $m=1$ and there exists a non-zero map $L|_F\rightarrow \cO_F$.
\end{lem}

\begin{proof}
	Consider the following composition of maps of vector bundles:
	\begin{equation}
		\label{eq.Map-Cotang-Tang}
		\Omega_X^1\otimes L\longrightarrow \cD^*\otimes L \longrightarrow \cD \longrightarrow T_X,
	\end{equation}
	where the first map is induced by tensoring the dual of the contact structure exact sequence with $L$ and the second map is the isomorphism induced by the non-degenerate pairing $\wedge^2 \cD \rightarrow L$. Composing it with the cotangent map of $\phi$ tensored with $L$ then yields a map
	\[
	\alpha\colon \phi^*\Omega_Y^1\otimes L \longrightarrow T_X,
	\]
	and the required map $\beta$ is obtained by restricting $\alpha$ to $F$.
	
	If $\phi$ has maximal rank at some point of $F$, then the cotangent map $\phi^*\Omega_Y^1|_F\rightarrow \Omega^1_X|_F$ is generically injective. Since the kernel of the map $\Omega_X^1\otimes L\rightarrow \cD^*\otimes L$ is isomorphic to $\cO_X$, it follows that if $\beta=0$, then there exists a factorization 
	\[
	\phi^*\Omega_Y^1\otimes L|_F\cong L|_F^{\oplus m} \longrightarrow \cO_F\cong \cO_X|_F \longrightarrow \Omega_X^1\otimes L|_F\longrightarrow T_X|_F,
	\]
	which then implies that $m=1$ and $L|_F\rightarrow \cO_F$ is non-zero.
\end{proof}

Let $F$ be a general fiber of the MRC quotient $h\colon X\dashrightarrow T$. Then $F$ is a rationally connected projective manifold. Let $\sV_F$ be the closed analytic subspace of $\sV$ corresponding to the maps whose images are contained in $F$. Then $\sV_F$ is a closed subvariety of $\Hom_{\bir}(\bP^1,F)$ and there exists an irreducible component $\sK$ of $\sV_F$, which is an unsplit covering family of rational curves on $F$. Let $f\colon \bP^1 \rightarrow F$ be a general member of $\sK$. By \cite[Corollary 2.9]{Kollar1996}, there exists an integer $0\leq p\leq \dim F-1=2n-m$ such that 
\begin{equation}
	\label{eq.Minimal-Splitting}
	f^*T_F\cong \cO_{\bP^1}(2)\oplus \cO_{\bP^1}(1)^{\oplus p} \oplus \cO_{\bP^1}^{\oplus (\dim F-p-1)}.
\end{equation}

\begin{lem}
	\label{lem.Rank-p}
	$p=n-1$.
\end{lem}

\begin{proof}
	It follows from the following exact sequence
	\[
	0\longrightarrow f^*T_F \longrightarrow f^*T_X \longrightarrow f^*N_{F/X}\cong \cO_{\bP^1}^{\oplus m} \longrightarrow 0
	\]
	that
	\[
	f^*T_X\cong f^*T_F \oplus \cO_{\bP^1}^{\oplus m} \cong \cO_{\bP^1}(2)\oplus \cO_{\bP^1}(1)^{\oplus p} \oplus \cO_{\bP^1}^{\oplus(2n-p)}.
	\]
	Now we conclude by Lemma \ref{lem.Locus-Rational-Curves} as $\rank^+ f^*T_X=n$.
\end{proof}

Since $\sK$ is a covering family of rational curves on $F$ and $L\cdot \sK=1$, the map $\beta$ from \eqref{eq.beta} is not the zero map and furthermore $\beta$ factors through $T_F\rightarrow T_X|_F$. Let $\cF\subset T_F$ be the image of $\beta$. Then $\cF$ is torsion free and hence locally free in codimension one. In particular, for a general member $f\colon \bP^1 \rightarrow F$ in $\sK$, the sheaf $\cF$ is locally free along $f(\bP^1)$ by \cite[II, Proposition 3.7]{Kollar1996}. It follows that $f^*\cF$ is ample since $f^*\cF$ is a quotient of the ample vector bundle $f^*L^{\oplus m}\cong \cO_{\bP^1}(1)^{\oplus m}$.  Then applying \cite[Proposition 2.7]{AraujoDruelKovacs2008} to the triple $(F,\sK,\cF)$ yields:

\begin{prop}
	\label{prop.Pn-bundle}
	There exists a dense open subset $F^{\circ}$ of $F$ and a $\bP^n$-bundle 
	\[
	\mu \colon F^{\circ}\rightarrow Z^{\circ}
	\]
	such that 
	\begin{enumerate}
		\item $\cF|_{F^{\circ}}\subset T_{F^{\circ}/Z^{\circ}}$, and
		
		\item every rational curve $[f]\in \sK$ meeting $F^{\circ}$ is a line contained in a fiber of $F^{\circ}\rightarrow Z^{\circ}$, and 
		
		\item for a general point $x\in F^{\circ}$, there exists a linear subspace $P\cong \bP^r$ contained in the fibers of $\mu$ such that $P$ is tangent to $\cF$ along a dense open subset, where $r=\rank \cF$.
	\end{enumerate}
\end{prop}

\begin{proof}
	The existence of $F^{\circ}$ and a projective space-bundle $F^{\circ}\rightarrow Z^{\circ}$ follows immediately from \cite[Proposition 2.7]{AraujoDruelKovacs2008} and then it follows from \eqref{eq.Minimal-Splitting} and Lemma \ref{lem.Rank-p} that the dimension of the fibers of $\mu$ is equal to $n$ .
	
	Finally notice that the restriction of $L$ to the fibers of $\mu$ is ample, so the restriction of $\cF$ to a general fiber $S\cong \bP^n$ of $\mu$ is indeed an ample subsheaf of $T_S$ and hence isomorphic to either $\sO_{\bP^n}(1)^{\oplus r}$ or $T_{\bP^n}$ by \cite[Theorem 1.1]{Liu2019}; that is, the restriction $\cF|_S$ is the algebraically integrable foliation on $S\cong \bP^n$ defined by a linear projection $\bP^n\dashrightarrow \bP^{n-r}$ (cf.~\cite[Example 3.4]{Liu2019}), as required.
\end{proof}

\begin{cor}
	\label{cor.irred-Vx}
	For a general point $x\in X$, the variety $\sV_x$ is irreducible.
\end{cor}

\begin{proof}
	Let $F$ be the fiber of $h$ containing $x$. Let $\sK$ be an irreducible component of $\sV_F$ dominating $F$. As $\dim \Locus \sV(0\mapsto x)=n$ by Lemma \ref{lem.Degree-Locus}, it follows from Proposition \ref{prop.Pn-bundle} that $\sK_x$ parametrizes exactly the lines contained in the fiber of $\mu\colon F^{\circ}\longrightarrow Z^{\circ}$ passing through $x$. In particular, the variety $\sK_x$ is irreducible.
	
	Now we assume to the contrary that $\sV_x$ is not irreducible. Then there exist two irreducible components $\sK$, $\widetilde{\sK}$ of $\sV_F$ dominating $F$ such that $\sK_x\not=\widetilde{\sK}_x$. Denote by $\mu\colon F^{\circ}\rightarrow Z^{\circ}$ and $\widetilde{\mu}\colon \widetilde{F}^{\circ}\rightarrow \widetilde{Z}^{\circ}$ the associated $\bP^n$-bundles provided in Proposition \ref{prop.Pn-bundle}, respectively. In particular, we may assume that $x\in F^{\circ}\cap \widetilde{F}^{\circ}$ and there exist linear subspaces
	\[
	P\subset F^{\circ}_{\mu(x)}\coloneqq \mu^{-1}(\mu(x))\quad \textup{and}\quad \widetilde{P}\subset \widetilde{F}^{\circ}_{\widetilde{\mu}(x)}\coloneqq \widetilde{\mu}^{-1}(\widetilde{\mu}(x)),
	\]
	which are tangent to $\cF$ at $x$. However, by the uniqueness of the leaves of $\cF$ through $x$, we must have $P=\widetilde{P}$. As a consequence, the lines contained in $F^{\circ}_{\mu(x)}$ and $\widetilde{F}^{\circ}_{\widetilde{\mu}(x)}$ are numerically equivalent as $1$-cycles in $F$. This implies $F^{\circ}_{\mu(x)}=\widetilde{F}^{\circ}_{\widetilde{\mu}(x)}$ and hence $\sK_x=\widetilde{\sK}_x$, which is absurd.
\end{proof}

\section{Proof of Theorem \ref{thm.MainThm}} 

Throughout this section, we retain the notation from \S\,\ref{s.Def-Rat-Curve}. Let $(X,L)$ be a non-projective compact K\"ahler contact manifold of dimension $2n+1$ and let  $\sV$ be an unsplit family of rational curves containing an extremal rational curve on $X$. Note that Proposition \ref{prop.Pn-bundle} holds for any general fiber of the MRC quotient $h\colon X\dashrightarrow T$. Since $\sD(X)$ has only countably many irreducible components \cite{Fujiki1979}, we can find an irreducible compact analytic subspace $\sW$ of $\sD(X)$ such that 
\begin{itemize}
	\item the universal cycle $\sU\subset \sW\times X$ over $\sW$ dominates $X$, and 
	
	\item\label{i2.W} the subset of points in $\sW$ parametrizing the disjoint $\bP^n$'s contained in general fibers of $h$ (given by Proposition \ref{prop.Pn-bundle}) is dense in $\sW$.
\end{itemize}
Let $p\colon \sU\rightarrow X$ be the evaluation map. Since $\sW$ is compact and $q\colon \sU\rightarrow \sW$ is proper, the complex space $\sU$ is again compact, and therefore $p$ is a proper map. For a point $w\in \sW$, we denote by $\sU_w=q^{-1}(w)$ the fiber of $q$ over $w$. Clearly, the general fiber $\sU_w$ of $q$ is isomorphic to $\bP^n$ and all the lines contained in $\sU_w$ are parametrized by $\sV$.

\begin{lem}
	\label{lem.Fibers-q}
	For any $w\in \sW$ and any $x\in\sU_w$, the fiber $\sU_w$ is irreducible and $p(\sU_w)\subset \Locus \sV(0\mapsto p(x))$.
\end{lem}

\begin{proof}
	Let $y\in \sU_w$ be an arbitrary point. Then there exists a one-dimensional complex subspace $\sT$ contained in the analytic germ of $\sW$ at $w$, together with points $x_t,y_t\in \sU_t$, $t\in \sT$, such that 
	\[
	\lim_{t\to w} (\sU_t,x_t,y_t) = (\sU_w,x,y).
	\]
	Moreover, we may assume that $\sU_t\cong \bP^n$ for $t\not=w$. Let $l_t$ be the line parametrized by $\sV$ joining $x_t$ and $y_t$ for $t\not= w$. Since $\sV$ is unsplit, the limit 
	\[
	l_w=\lim_{t\to w} l_t
	\]
	still lies in $\sV$ so that $l_w$ is an irreducible rational curve passing through $x$ and $y$. Hence the fiber $\sU_w$ is irreducible and $p(\sU_w) \subset \Locus \sV(0\mapsto p(x))$
\end{proof}

\begin{lem}
	\label{lem.p-iso}
	The map $p$ is an isomorphism.
\end{lem}

\begin{proof}
	Firstly we prove that $p$ is finite. Assume to the contrary that $p$ is not finite. Then one can find a point $x\in X$ such that $p^{-1}(x)$ is positive-dimensional. In particular, there exists an irreducible positive-dimensional analytic subspace $\sW_x$ of $\sW$ consisting of cycles passing through $x$. Note that $\dim p(q^{-1}(\sW_x))\geq n+1$ and then Lemma \ref{lem.Fibers-q} above implies
	\[
	\dim \Locus \sV(0\mapsto x) \geq \dim p(q^{-1}(\sW_x)) \geq n+1,
	\]
	which contradicts Lemma \ref{lem.Locus-Rational-Curves}. Hence $p$ is finite.
	
	Next we show that $p$ is an isomorphism. Since $X$ is smooth, by the analytic version of Zariski Main Theorem, it is enough to prove that $p$ is birational. Assume to the contrary that $p$ is not birational. Let $x\in X$ be a general point. Then there exist two points $w\not=w'\in \sW$ such that $\sU_w\cong \sU_{w'}\cong \bP^n$ and $x\in p(\sU_w)\cap p(\sU_{w'})$. As $\dim \Locus\sV(0\mapsto x)=n$ and $p(\sU_w)\cup p(\sU_{w'})\subset \Locus \sV(0\mapsto x)$, it follows that $\sV_x$ is not irreducible, which contradicts Corollary \ref{cor.irred-Vx}.
\end{proof}

\begin{proof}[Proof of Theorem \ref{thm.MainThm}]
	By Lemmas \ref{lem.Fibers-q} and \ref{lem.p-iso}, there exists an equi-dimensional fibration $q\colon X\rightarrow Y$ with irreducible fibers such that the general fiber $F$ is isomorphic to $\bP^n$ and $L|_F\cong \cO_{\bP^n}(1)$ (see Lemma \ref{lem.Degree-Locus} and Proposition \ref{prop.Pn-bundle}). Moreover, after taking the Stein factorization, we may also assume that $Y$ is normal. In particular, since $X$ is smooth and $L$ is $q$-big, the map $q$ is projective by \cite[Theorem 3.1]{ClaudonHoering2024}.
	
	Let $y\in Y$ be an arbitrary point and let $X_y$ be the fiber of $q$ over $y$ with $\nu\colon \widetilde{X}_y\rightarrow X_y$ its normalization. Then both $X_y$ and $\widetilde{X}_y$ are projective. Moreover, for any point $x\in X_y$, as $\dim \Locus \sV(0\mapsto x)=n$ by Lemma \ref{lem.Degree-Locus}, it follows from Lemma \ref{lem.Fibers-q} that $X_y$ is an irreducible component of $\Locus\sV(0\mapsto x)$. In particular, note that every rational curve passing through a general point $x\in X_y$ can be lifted to $\widetilde{X}_y$, so there exists a generically unsplit family $\sK$ of rational curves on $\widetilde{X}_y$ such that $\Locus \sK_x=\widetilde{X}_y$ for any general point $x\in \widetilde{X}_y$ and $\nu^*L\cdot \sK=1$. Then $\widetilde{X}_y\cong \bP^n$ and $\nu^*L\cong \cO_{\bP^n}(1)$ by \cite[Theorem 3.6]{Kebekus2002}. In particular, the contact line bundle $L$ is $q$-ample.

	Now we can apply \cite[Lemma 2.12]{Fujita1987} to conclude that $Y$ is smooth and consequently $(X,L)$ is isomorphic to a projective bundle $(\bP E,\sO_{\bP E}(1))$ by \cite[Corollary 5.4]{Fujita1975}, where $E\rightarrow Y$ is a vector bundle of rank $n+1$. Here we remark that Fujita's results are stated only for projective varieties, the arguments carry over verbatim to projective morphisms. It follows that $Y$ is indeed a K\"ahler manifold, and then the same argument as in the last paragraph of the proof of \cite[Theorem 2.12]{KebekusPeternellSommeseWisniewski2000} shows that $E$ is isomorphic to $T_Y$.
\end{proof}

	\bibliographystyle{alpha}
	\bibliography{CKM}
\end{document}